\documentclass[11pt]{article}
\usepackage{mathrsfs}
\usepackage{pst-all,amsmath,amsthm}      
\usepackage{pst-poly,subfigure}     
\usepackage{multido}      
\usepackage{pstricks,pst-node,pst-tree}

\newtheorem{theorem}{Theorem}[section]
\newtheorem{lemma}{Lemma}[section]

\title{\Large On the spectral radii of bicyclic graphs with fixed independence number}
\author{{ \small {Xi-Ying Yuan\thanks{xiyingyuan2007@hotmail.com (X.Y. Yuan).  } }
 \quad }\\
\parbox[c]{10cm}{\vspace{0.5cm}\centering \small{ Department of Mathematics, Shanghai University, Shanghai, 200444,
China}\\  } }

\date{}
\begin{document}
\maketitle
 \maketitle
 \noindent{\bf \it Abstract}

 {Bicyclic graph is a connected graph in which the number
of edges equals the number of vertices plus one.  In this paper, we
determine the graph which alone maximizes the  spectral radii among
all the bicyclic graphs on $n$ vertices with fixed independence
number.}

\medskip


\section{ Introduction}
Let $G$ be a  simple graph. Denote by $N_{G}(v)$ (or simply $N(v)$)
the set of all the neighbors of a vertex $v$ in $G$,  and by
$d_{G}(v)$ (or  $d(v)$) the degree of $v$.  Let $A(G)$ be the
adjacency matrix of $G$ and $\Phi(G;x)$ be the characteristic
polynomial \mbox{det}$(xI-A(G))$. Since $A(G)$ is a real symmetric
matrix, all of its eigenvalues are real. The largest  eigenvalue of
$A(G)$ is called the spectral radius of $G$, denoted by $\rho(G)$.
When $G$ is connected, $A(G)$ is an irreducible matrix. And by the
Perron-Frobenius Theorem $\rho(G)$ has multiplicity one and there
exists a unique unit positive eigenvector corresponding to
$\rho(G)$.  We shall refer to such an eigenvector as the Perron
vector of $G$. Let $x$  be the Perron vector of a connected graph
$G$, and we always use $x_{u}$ to denote
 the coordinate of $x$ corresponding to the vertex $u$ of $G$.

\medskip
Brualdi and Solheid \cite{Brualdi} proposed the following general
problem, which became one of the classic problems of  spectral graph
theory:

 {\it Given a set  of graphs, find an upper
bound for the spectral radius  and characterize the graphs in which
the maximal spectral radius is attained.}

A subset $S$ of $V(G)$ is called an independent set of $G$ if no two
vertices in $S$ are adjacent in $G$. The independence number of $G$,
denoted by $\alpha(G)$, is the size of a maximum independent set of
$G$. We use the notations in  \cite{West}. Denote by
$\alpha^{\prime}(G)$ the edge independence number(or matching
number), by $\beta(G)$  the vertex covering number, and
$\beta^{\prime}(G)$ the  edge covering number for graph $G$.   For a
tree $T$ on $n$ vertices $\alpha(T)=n-\alpha^{\prime}(T)$ (see
Lemmas \ref{relation1},  \ref{relation2}). In \cite{Hou} the tree
with the maximal spectral radius among all the trees on $n$ vertices
with fixed matching number was determined. Thus the tree with the
maximal spectral radius among all the  trees on $n$ vertices with
fixed independence number was also determined. In \cite{Feng} the
graph with the maximal spectral radius among all the unicyclic
graphs on $n$ vertices with fixed independence number was
determined.

Here we are interested in finding the graph with the maximal
spectral radius among all the bicyclic graphs on $n$ vertices with
fixed independence number. We mainly  prove the following results.

 \begin{theorem}
Let $F(n,\frac{n-2}{2})$ and $M(n,\alpha)$ be the graphs as shown in
Fig. 2 and Fig. 3. For a   bicyclic graph $G$  on $n\,(n\geq10)$
vertices then

\medskip

 (1).  $\alpha(G)\geq\frac{n-2}{2}$;

\medskip
 (2).  if $\alpha(G)=\frac{n-2}{2}$, then
 $\rho(G)\leq\rho(F(n,\frac{n-2}{2}))$, where  $\rho(F(n,\frac{n-2}{2}))$ is the largest root
 of the equation $$x^{4}-2x^{3}-(n/2+1)x^2+nx+3=0,$$ and  equality holds if and only if
$G=F(n,\frac{n-2}{2})$;

\medskip
 (3). if $\alpha(G)\geq \frac{n-1}{2}$, then
 $\rho(G)\leq\rho(M(n,\alpha))$, where  $\rho(M(n,\alpha))$ is the largest root
 of the equation $$x^4-(\alpha+3)x^2-4x+(2\alpha-n+1)=0,$$ and  equality holds if and only if
$G=M(n,\alpha)$.
 \end{theorem}

\section{Preliminaries}

 The the following two lemmas are the main tools for some proofs in  later sections.

 \begin{lemma}
  \cite{wu}
  \label{deprive l}
  Let $u, v$ be two vertices of a connected graph $G$.
Suppose $v_{1}, v_{2},\cdots, v_{s} \,(1 \leq s \leq d(v)) $ are
some vertices in $N(v)\backslash (N(u)\bigcup\{u\}) $. Let $x$  be
the  Perron vector  of $G$. If $x_{u} \geq x_{v}$, let $G^{*}$ be
the graph obtained from $G$ by deleting the edges $v v_{1}, v v_{2},
\cdots, v v_{s}$ and adding the edges $u v_{1}, u v_{2}, \cdots, u
v_{s}$, then we have $\rho(G^{*})> \rho(G)$.
\end{lemma}

\begin{lemma}
\cite{Li} \label{grafting}
 Let $v$ be a vertex in a non-trivial
connected graph $G$ and suppose that two paths of lengths $k,m(k
\geq m \geq 1)$ are attached to $G$ by their end vertices at $v$ to
form $G_{k,m}$. Then $\rho(G_{k,m})
> \rho(G_{k+1,m-1})$.
\end{lemma}

\begin {lemma}
\cite{Gutman} \label{bound} For any simple graph  $G$  we have
$\rho(G)\geq \sqrt{\Delta(G)}$ holds.
\end{lemma}

\begin{lemma}
\label{relation1}
 \cite{}
  Let $G$ be a graph on $n$ vertices without
isolated vertices. Then
$$\alpha(G)+\beta(G)=\alpha^{\prime}(G)+\beta^{\prime}(G)=n.$$
\end{lemma}

\begin{lemma}
\label{relation2}
 \cite{}
  Let $G$ be a bipartite graph  without
isolated vertices. Then $\alpha(G)=\beta^{\prime}(G).$
\end{lemma}

\begin {lemma}
\cite{Cve} \label{formular1} Let $v$ be a vertex of $G$, and ${\cal
C}(v)$
 be the set of all cycles containing $v$. Then
$$\Phi(G;x)=x\Phi(G-v;x)-\sum_{u \in N(v)}\Phi(G-u-v;x)-2\sum_{Z \in {\cal C}(v)}\Phi(G- V(Z);x).$$
\end{lemma}

Let $C_{p}$ and $C_{q}$ be two vertex-disjoint cycles. Suppose that
$v_{1}$
 is a vertex of $C_{p}$ and $v_{\ell}$ is a
vertex of $C_{q}$. Joining $v_{1}$ and $v_{\ell}$ by a path
$v_{1}v_{2 } \cdots v_{\ell}$ on $\ell$ vertices, where $\ell\geq 1$
and $\ell= 1$ means identifying $v_{1}$ with $v_{\ell}$, the
resulting graph (see Fig.1), denoted by $B(p, \ell, q)$, is called
an $\infty$-graph. Let $P_{\ell+2}, P_{p+2}$ and $P_{q+2}$ be three
vertex-disjoint paths, where $0 \leq \ell \leq p\leq q $ and at most
one of them is 0. Identifying the three initial vertices and
terminal vertices of them, respectively, the resulting graph (see
Fig.1), denoted by $P(\ell,p, q)$, is called a $\theta$-graph.
 Obviously ${\cal B}(n)$ consists of two types of graphs:
one type, denoted by ${ \cal B}_{1}(n)$, are those graphs each of
which is an $\infty$-graph or an $\infty$-graph with trees attached;
the other type, denoted by ${ \cal B}_{2}(n)$, are those graphs each
of which is a $\theta$-graph or a $\theta$-graph with trees
attached.

\vspace{1cm}
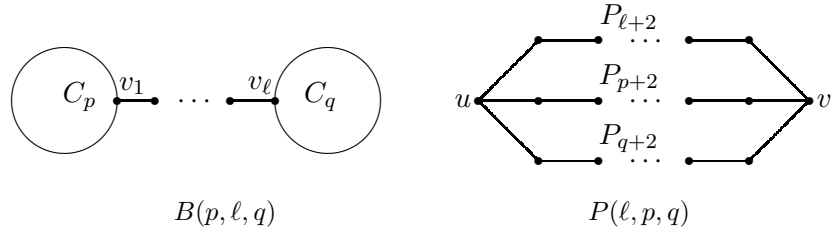
\begin{center}\unitlength 1mm
\linethickness{0.4pt}
\begin{picture}(65,-35)

\put(-15,5){\circle{16}} \put(3,3) {\makebox(-32,5)[cc]{$C_{p}$}}
\put(-8,5){\circle*{1}} \put(-6,7){\makebox(0,0)[cc]{$v_{1}$}}
\put(-3,5){\circle*{1}}

\put(0,4){$\cdots$}\put(7,5){\circle*{1}}\put(13,5){\circle*{1}}\put(20,5){\circle{16}}
\put(11,7){\makebox(0,0)[cc]{$v_{\ell}$}} \put(3,3)
{\makebox(32,5)[cc]{$C_{q}$}}
\qbezier(-8,5)(-8,5)(-3,5)\qbezier(7,5)(7,5)(13,5)
\put(5,-10){\makebox(0,0)[cc]{{\small \bf \quad  $B(p,\ell,q)$ }}}

\put(40,5){\circle*{1}}

 \put(38,5){\makebox(0,0)[cc]{$u$}}
\qbezier(40,5)(40,5)(48,5)\qbezier(40,5)(40,5)(48,13)\qbezier(40,5)(40,5)(48,-3)
\qbezier(56,13)(56,13)(48,13)\qbezier(56,-3)(56,-3)(48,-3)\qbezier(56,5)(56,5)(48,5)
 \put(48,5){\circle*{1}}\put(56,5){\circle*{1}}
\put(60,4){$\cdots$} \put(60,8){\makebox(0,0)[cc]{$P_{p+2}$}}
\put(68,5){\circle*{1}}\put(76,5){\circle*{1}}

\put(48,13){\circle*{1}}\put(56,13){\circle*{1}}
\put(60,12){$\cdots$}\put(68,13){\circle*{1}}
\put(76,13){\circle*{1}}

\put(60,16){\makebox(0,0)[cc]{$P_{\ell+2}$}}

\put(48,-3){\circle*{1}}\put(56,-3){\circle*{1}}
\put(60,-4){$\cdots$}\put(68,-3){\circle*{1}}\put(76,-3){\circle*{1}}

\put(60,0){\makebox(0,0)[cc]{$P_{q+2}$}} \put(84,5){\circle*{1}}
 \put(86,5){\makebox(0,0)[cc]{$v$}}

\qbezier(84,5)(84,5)(76,5)\qbezier(84,5)(84,5)(76,13)\qbezier(84,5)(84,5)(76,-3)
\qbezier(68,5)(68,5)(76,5)\qbezier(68,13)(68,13)(76,13)\qbezier(68,-3)(68,-3)(76,-3)

\put(60,-10){\makebox(0,0)[cc]{{\small \bf  \quad $P(\ell,p,q)$ }}}
\put(25,-20){\makebox(0,0)[cc]{{\small \bf \quad  Fig. 1 the graphs
$B(p,\ell,q)$ and $P(\ell,p,q)$}}}

 \end{picture}
\end{center}

\vspace{2cm}

 The base of a bicyclic graph $G$, denoted by
$\widehat{G}$, is the (unique) minimal bicyclic subgraph of $G$.  We
use $V_{c}(G)$ to denote all the vertices  on the cycles of a graph
$G$.

 \begin{lemma}
 \label{bound for elfa}
Let $G$ be a graph in ${\cal B}(n)$. Then

(1). $\alpha(G)\geq\frac{n-2}{2}$;

(2). $\alpha(G)=\frac{n-2}{2}$ if and only if
$\widehat{G}=B(p,\ell,q)$ for some three integers $p,\ell,q$, where
$\ell\geq2$,  $p,q$ are odd, and the graph $G-V_{c}(G)$ has a
perfect matching.
 \end{lemma}

\begin{proof}
(1). Let $G$ be a graph in ${\cal B}(n)$. Then $\widehat{G}$ is  an
$\infty$-graph, or  a $\theta$-graph. When $\widehat{G}=B(p,\ell,q)$
for some three integers $p,\ell,q$, where $\ell\geq2$, and $p,q$ are
odd, let $v_{1}$ be a vertex on the cycle $C_{p}$, and $v_{\ell}$ be
a vertex on the cycle $C_{q}$. Then $G-v_{1}-v_{\ell}$ is a forest,
and so $\alpha(G)\geq\alpha(G-v_{1}-v_{\ell})\geq\frac{n-2}{2}$. For
other cases we may always choose a proper vertex of $G$, say $v$,
such that  $G-v$ is a bipartite graph, and then
$\alpha(G)\geq\alpha(G-v)\geq\frac{n-1}{2}$.

\medskip

(2). From the proof of (1) we know that if
$\alpha(G)=\frac{n-2}{2}$, then  $\widehat{G}=B(p,\ell,q)$, where
$\ell\geq2$, $p,q$ are odd. Now we prove that the graph $G-V_{c}(G)$
has a perfect matching. Let $$G-V_{c}(G)=T_{1}\bigcup \cdots \bigcup
T_{s},$$ where $T_{i}$ is a tree for each $i=1,\cdots,s$. Suppose to
the contrary that $T_{1}$ has no perfect matching. Write
$|V(T_{1})|=t$, then $\alpha^{\prime}(T_{1})\leq\frac{t-1}{2}$. By
K$\ddot{o}$nig-Egverary theorem we have
$$\alpha(T_{1})=\beta^{\prime}(T_{1})=t-\alpha^{\prime}(T_{1})\geq\frac{t+1}{2}.$$
Let $S_{1}$ be an independent set of $T_{1}$ with
$|S_{1}|=\alpha(T_{1})$.  Let $u$ be the vertex on the cycle and $u$
has a neighbour in $T_{1}$, and $v$ be a vertex on another cycle of
$G$. Then $G-u-v-V(T_{1})$ is a forest. Let $S_{2}$ be an
independent set of $G-u-v-V(T_{1})$ with
$|S_{2}|=\alpha(G-u-v-V(T_{1}))\geq\frac{n-t-2}{2}$. The fact that
$u \not \in (S_{1}\bigcup S_{2})$ insures that $S_{1}\bigcup S_{2}$
is an independent set of $G$.  Thus $\alpha(G)\geq|S_{1}\bigcup
S_{2}|\geq \frac{n-1}{2}$. This contradicts the hypothesis that
$\alpha(G)=\frac{n-2}{2}$.

\medskip
Now we prove the sufficiency for (2).

Write $|V_{c}(G)|=k$. If $\widehat{G}=B(p,\ell,q)$, where
$\ell\geq2$,  $p,q$ are odd, then any independent set of $G$
contains at most $\frac{k-2}{2}$ vertices in $V_{c}(G)$. And if the
graph $G-V_{c}(G)$ has a perfect matching, then any independent set
of $G$ contains at most $\frac{n-k}{2}$ vertices outside of
$V_{c}(G)$. Thus $\alpha(G)\leq\frac{n-2}{2}$. And we have proved
that $\alpha(G)\geq\frac{n-2}{2}$. So $\alpha(G)=\frac{n-2}{2}$.
\end{proof}

Let $${\cal B}(n,\alpha)=\{G\,| \,G \in {\cal B}(n),
\alpha(G)=\alpha\},$$ from Lemma \ref{bound for elfa} we know that
$\alpha\geq\frac{n-2}{2}$. In Section 3 we will determine the graph
with maximal spectral radius in  ${\cal B }(n,\frac{n-2}{2})$. When
$\alpha\geq \frac{n-1}{2}$ the  graph with maximal spectral radius
in ${\cal B }(n,\alpha)$ will be determined in Section 4.

\section{The  graph with maximal spectral radius in  ${\cal B
}(n,\frac{n-2}{2})$}

Let  $F(n,\frac{n-2}{2})$ be the graph as shown in Fig.2. In this
section we will prove that  $F(n,\frac{n-2}{2})$ alone maximizes the
spectral radius among the graphs in ${\cal B}(n,\frac{n-2}{2})$ when
$n\geq 10$.

\vspace{2cm}
\begin{center}
\unitlength 1mm
\linethickness{0.4pt}
\begin{picture}(60,-35)
\put(-25,5){\circle*{1}}
 \put(-25,-5){\circle*{1}}
\put(-15,0){\circle*{1}} \put(-20,8){\circle*{1}}
\put(-17,8.5){$\cdots$} \put(-20,16.5){$\overbrace{\hspace{1cm}}$}
\put(-15,21.5){\makebox(0,0)[cc]{$_{\frac{n-6}{2}}$}}

 \put(-10,8){\circle*{1}}
\put(-20,16){\circle*{1}}
 \put(-10,16){\circle*{1}}
\put(-5,0){\circle*{1}}
\put(5,5){\circle*{1}}
 \put(5,-5){\circle*{1}}

 \qbezier(-15,0)(-15,0)(-25,5)
  \qbezier(-15,0)(-15,0)(-25,-5)
   \qbezier(-15,0)(-15,0)(-20,8)
    \qbezier(-15,0)(-15,0)(-10,8)
     \qbezier(-15,0)(-15,0)(-5,0)
\qbezier(-25,-5)(-25,-5)(-25,5) \qbezier(5,-5)(5,-5)(5,5)
 \qbezier(-20,16)(-20,16)(-20,8)
    \qbezier(-10,16)(-10,16)(-10,8)
\qbezier(5,5)(5,5)(-5,0)\qbezier(5,-5)(5,-5)(-5,0)

\put(43,8.5){$\cdots$} \put(42,16.5){$\overbrace{\hspace{1cm}}$}
\put(47,21.5){\makebox(0,0)[cc]{$_{\frac{n-8}{2}}$}}

\put(-10,-10){\makebox(0,0)[cc]{$F(n,\frac{n-2}{2})$}}
\put(25,5){\circle*{1}}
 \put(25,-5){\circle*{1}}
 \put(35,0){\circle*{1}}
 \put(45,0){\circle*{1}}
 \put(55,0){\circle*{1}}
\put(35,8){\circle*{1}}
\put(42,8){\circle*{1}}
\put(52,8){\circle*{1}}
\put(42,16){\circle*{1}}
\put(52,16){\circle*{1}}
 \put(65,5){\circle*{1}}
 \put(65,-5){\circle*{1}}
\qbezier(45,0)(45,0)(35,0)
 \qbezier(45,0)(45,0)(55,0)
 \qbezier(45,0)(45,0)(35,8)
\qbezier(45,0)(45,0) (42,8) \qbezier(45,0)(45,0)(52,8)
\qbezier(25,5)(25,5)(35,0) \qbezier(25,-5)(25,-5)(35,0)
 \qbezier(65,5)(65,5)(55,0) \qbezier(65,-5)(65,-5)(55,0)
\qbezier(25,5)(25,5)(25,-5)\qbezier(65,-5)(65,-5)(65,5)
 \qbezier(42,16)(42,16)(42,8)
  \qbezier(52,16)(52,16)(52,8)
\put(45,-10){\makebox(0,0)[cc]{$F^{\prime}$}}

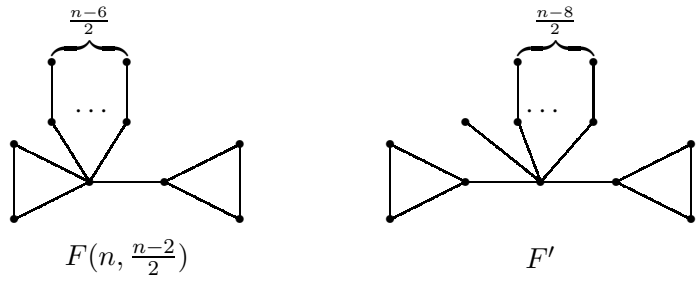
\put(20,-20){\makebox(0,0)[cc]{Fig.2 \quad  the graphs
$F(n,\frac{n-2}{2})$ and $F^{\prime}$}}
\end{picture}
\end{center}
\vspace{2.5cm}

\begin{lemma}
\label{juti1}
 Let $F^{\prime}$ and $F(n,\frac{n-2}{2})$ be the
graphs as shown in Fig.2. Then
$\rho(F(n,\frac{n-2}{2}))>\rho(F^{\prime})$.
\end{lemma}

\begin{proof}
Write $n=2c$. By using Lemma \ref{formular1} and tedious
calculations we have
$$
\Phi(F(n,\frac{n-2}{2});x)=(x^{2}-1)^{c-3}(x+1)^{2}[x^{4}-2x^{3}-(c+1)x^2+2cx+3],$$
$$
\Phi(F^{\prime};x)=(x^{2}-1)^{c-5}(x+1)^{4}(x-2)(x^{5}-2x^{4}-cx^{3}+2cx^2-x-2).
$$
Set $$g(x)=(x-1)^{2}[x^{4}-2x^{3}-(c+1)x^2+2cx+3],
$$
and
$$ h(x)=(x-2)(x^{5}-2x^{4}-cx^{3}+2cx^2-x-2),
$$
then $\rho(F(n,\frac{n-2}{2}))$ is  the largest root of the equation
$g(x)=0$, and  $\rho(F^{\prime})$ is  the largest root of the
equation $h(x)=0$.  When $n\geq 10$, i.e., $c\geq 5$ we have
$\rho(F^{\prime})\geq\sqrt{\Delta(F^{\prime})}>2$, and  it may be
verified that
$$h(x)-g(x)=(c-3)x(x-2)+1,$$ then $g(\rho(F^{\prime}))<0$. Thus the
largest root of the equation
  $g(x)=0$ is larger than $\rho(F^{\prime})$, i.e., $\rho(F(n,\frac{n-2}{2}))>\rho(F^{\prime})$.
\end{proof}

\begin{theorem}
Graph $F(n,\frac{n-2}{2})$ alone maximizes the spectral radius among
the graphs in ${\cal B}(n,\frac{n-2}{2})$ when $n\geq 10$.
\end{theorem}

\begin{proof}
Suppose $G^{*}$ is a graph with maximal spectral radius among the
graphs in ${\cal B}(n,\frac{n-2}{2})$. From (2) of Lemma \ref{bound
for elfa} we know that $\widehat{G^{*}}=B(p,\ell,q)$ for some three
integers $p,\ell,q$, where  $\ell\geq2$, $p,q$ are odd, and the
subgraph $G^{*}-V_{c}(G^{*})$ has a perfect matching.  Denote by
$v_{1}v_{2}\cdots v_{\ell-1}v_{\ell}$ the path joining the cycles
$C_{p}$ and $C_{q}$, where $v_{1}$ lies on $C_{p}$ and  $v_{\ell}$
lies on $C_{q}$. Let $x$ be the Perron vector of $G^{*}$.

\medskip
{\bf Claim 1.} Any  vertex in
$V_{c}(G^{*})\backslash\{v_{1},v_{\ell}\}$
 has degree 2.
\medskip

{\bf Proof of Claim 1.} Suppose to the contrary that there exists a
vertex in $V_{c}(G^{*})\backslash\{v_{1},v_{\ell}\}$, say $w$, with
degree at least 3. Without loss of generality  assume that $w$ lies
on $C_{p}$. Let $w^{\prime}$ be a neighbour of $w$ such that
$w^{\prime}\not \in V_{c}(G^{*})$.  Set
$$G^{\prime}= \left\{\begin{array}{ll}
G^{*}-ww^{\prime}+v_{1}w^{\prime}, &\, \mbox{if}\, \, x_{v_{1}}\geq x_{w}; \\
G^{*}-v_{1}v_{2}+ wv_{2},& \, \mbox{if}\, \,x_{w}>x_{v_{1}}.
\end{array}
\right.$$

Then $G^{\prime}-V_{c}(G^{\prime})$  also has a perfect matching,
furthermore $G^{\prime}$ is  in ${\cal B}(n,\frac{n-2}{2})$. While
we have $\rho(G^{\prime})>\rho (G^{*})$ from Lemma \ref{deprive l}.
This contradicts the definition of $G^{*}$.

\medskip
{\bf Claim 2.} $p=q=3$.
\medskip

{\bf Proof of Claim 2.}
\medskip
Suppose to the contrary that $p \geq 5$. Denote by
$C_{p}=v_{1}w_{1}w_{2}\cdots w_{p-1}w_{p}(=v_{1})$. Set
$$G^{\prime}= \left\{\begin{array}{ll}
G^{*}-w_{p-1}v_{1}+w_{2}v_{1}, &\, \mbox{if}\, \, x_{w_{2}}\geq x_{w_{p-1}}; \\
G^{*}-w_{2}w_{1}+ w_{p-1}w_{1},& \, \mbox{if}\,
\,x_{w_{p-1}}>x_{w_{2}}.
\end{array}
\right.$$ Then  $G^{\prime}$ is also in ${\cal B}(n,\frac{n-2}{2})$,
while  $\rho(G^{\prime})>\rho (G^{*})$.

 By comparing  the coordinates   $x_{v_{1}}$ and $x_{v_{\ell}}$ and using Lemma \ref{deprive l}, we
  may prove that at most one of $\{v_{1},v_{\ell}\}$ with degree more
than 3. Next we may suppose that $d(v_{1})\geq d(v_{\ell})$.

\medskip
{\bf Claim 3.} At most one   vertex outside of  $V_{c}(G^{*})$  has
degree more than 3.
\medskip

{\bf Proof of Claim 3.} Suppose to the contrary that there exist two
vertices, say $u,v$, outside of  $V_{c}(G^{*})$ with degree more
than 3.  Without loss of generality  assume that $x_{u}\geq x_{v}$.
Let $v^{\prime}$ be a neighbour of $v$ on the path between $u$ and
$v$, $v^{\prime\prime}$ be the vertex saturated by $v$ in a perfect
matching of $G^{*}-V_{c}(G^{*})$. Since $d(v)\geq3$, we may suppose
that $w\in (N(v)\backslash\{v^{\prime},v^{\prime\prime}\})$. Set
$G^{\prime}= G^{*}-vw+uw$. Then $G^{\prime}$ is also in ${\cal
B}(n,\frac{n-2}{2})$, while $\rho(G^{\prime})>\rho (G^{*})$.

\medskip
By using the similar arguments as the proof of Claim 3, we may prove
that if $d(v_{1})\geq 4$ then each vertex outside of  $V_{c}(G^{*})$
has degree at most 2.

\medskip
{\bf Claim 4.} $\ell \leq3$.
\medskip

{\bf Proof of Claim 4.}  Suppose to the contrary that $\ell \geq 4$,
then $v_{2}\not= v_{\ell-1}$. Set
$$G^{\prime}= \left\{\begin{array}{ll}
G^{*}-v_{\ell-1}v_{\ell}+v_{2}v_{\ell}, &\, \mbox{if}\, \, x_{v_{2}}\geq x_{v_{\ell-1}}; \\
G^{*}-v_{2}v_{1}+ v_{\ell-1}v_{1},& \, \mbox{if}\,
\,x_{v_{\ell-1}}>x_{v_{2}}.
\end{array}
\right.$$ Then  $G^{\prime}$ is also in ${\cal B}(n,\frac{n-2}{2})$,
while  $\rho(G^{\prime})>\rho (G^{*})$. Thus $\ell \leq3$.

Furthermore by using the above results and Lemma \ref{grafting} we
have if $\ell=2$,  then  $G^{*}=F(n,\frac{n-2}{2})$. If $\ell=3$,
then $G^{*}=F^{\prime}$, while from Lemma \ref{juti1} we know that
$\rho(F^{\prime})<\rho(F(n,\frac{n-2}{2}))$. Thus $\ell =2$ and we
have  $G^{*}=F(n,\frac{n-2}{2})$.
\end{proof}

\section{The  graph with maximal spectral radius in  ${\cal B
}(n,\alpha)$ when $\alpha\geq \frac{n-1}{2}$ }

 It is easy to see  that every  connected graph $G$ has at most
$\alpha(G)$ pendant vertices. In this section we may suppose
$\alpha\geq \frac{n-1}{2}$. Now we give a partition for the graphs
in ${\cal B }(n,\alpha)$ according to the number of the pendant
vertices.

\medskip

Class \,($C1$)\,: The graphs in   ${\cal B }(n,\alpha)$ with $k$
pendant vertices, where $k\leq \alpha -2$.

Class \,($C2$)\,: The graphs in   ${\cal B }(n,\alpha)$ with
$\alpha-1$ pendant vertices.

 Class \,($C3$)\,: The graphs in   ${\cal B }(n,\alpha)$ with
$\alpha$ pendant vertices.

We will discuss the spectral radii of the graphs in Class \,($C1$)
in Section 4.1, the spectral radii of the graphs in Class \,($C3$)
in Section 4.2, and the spectral radii of the graphs in Class
\,($C2$) in Section 4.3.

\subsection{The graphs in   ${\cal B }(n,\alpha)$ with $k$
pendant vertices and  $k\leq \alpha -2$.}

Let $B^{\sharp}(k)$ be the graph on $n$ vertices, obtained by
attaching $k$ paths of almost equal length to the vertex with degree
4 of $B(3,1,3)$. The following result was shown in
\cite{Guoshuguang} and

\begin{lemma}
(\cite{Guoshuguang, ?})
 \label{extremal graph by guo}
 Suppose $G$ is a bicyclic  graph  on $n$ vertices
with $k$ pendant vertices, then $\rho(G)\leq \rho(B^{\sharp}(k))$,
with equality if and only if $G = B^{\sharp}(k)$.
\end{lemma}

\vspace{0.5cm}
\begin{center}
\unitlength 1mm \linethickness{0.4pt}
\begin{picture}(60,-35)
\put(15,5){\circle*{1}}
 \put(15,-5){\circle*{1}}
 \put(25,0){\circle*{1}}
\put(20,8){\circle*{1}} \put(30,8){\circle*{1}} \put(23,4){$\cdots$}
\put(20,8.5){$\overbrace{\hspace{1cm}}$}
\put(25,12){\makebox(0,0)[cc]{$_{2\alpha-n+1}$}}

\put(20,8){\circle*{1}} \put(30,8){\circle*{1}}
\put(23,-14){$\cdots$} \put(20,-16.5){$\underbrace{\hspace{1cm}}$}
\put(25,-19.5){\makebox(0,0)[cc]{$_{n-\alpha-3}$}}
\put(20,-8){\circle*{1}} \put(30,-8){\circle*{1}}
\put(20,-16){\circle*{1}} \put(30,-16){\circle*{1}}

 \put(35,5){\circle*{1}}
\put(35,-5){\circle*{1}} \qbezier(25,0)(25,0)(15,5)
 \qbezier(25,0)(25,0)(15,-5)
\qbezier(25,0)(25,0) (35,5)\qbezier(25,0)(25,0)(35,-5)
\qbezier(25,0)(25,0)(20,8) \qbezier(25,0)(25,0)(30,8)
\qbezier(25,0)(25,0)(20,-8) \qbezier(25,0)(25,0)(30,-8)
 \qbezier(15,5)(15,5)(15,-5)\qbezier(35,5)(35,5)(35,-5)
\qbezier(20,-16)(20,-16)(20,-8) \qbezier(30,-16)(30,-16)(30,-8)

\put(25,-25){\makebox(0,0)[cc]{Fig.3 \quad the graph $M(n,\alpha)$}}

\end{picture}
\end{center}
\vspace{2.5cm}

\begin{theorem}
 Let $G$ be a graph in ${\cal B}(n,\alpha)$ with $k$ pendant vertices. When $k\leq \alpha-2$ and $\alpha\geq \frac{n-2}{2}$, we have
 $\rho(G)\leq \rho(M(n,\alpha))$ with equality  if and only if $G
 =M(n,\alpha)$.
 \end{theorem}

\begin{proof}
First by using  Lemma \ref{grafting} directly we have if  $1\leq k
\leq n-6$, then $\rho(B^{\sharp}(k))<\rho(B^{\sharp}(k+1))$. It is
easy to see that when $\alpha\geq \frac{n-1}{2}$, then
$B^{\sharp}(\alpha-2)=M(n,\alpha)$. Then from  Lemma \ref{extremal
graph by guo} we have $$\rho(G)\leq \rho(B^{\sharp}(k))\leq
\rho(B^{\sharp}(\alpha-2))=\rho(M(n,\alpha)).$$ Furthermore it is
not difficult to see that the quality holds if and only if $G
=M(n,\alpha)$.
\end{proof}

\subsection{The  graphs  in  ${\cal B }(n,\alpha)$ with $\alpha$ pendant vertices}

Set $${\cal B}(n,\alpha,\alpha)=\{G \,|\, G \in {\cal B}(n,\alpha)
\,\, \mbox{and}\,\, G \,\, \mbox{contains}\,\, \alpha
\,\,\mbox{pendant vertices}\}.$$ In this section we will prove that
the spectral radii of the graphs in  ${\cal B}(n,\alpha,\alpha)$ are
less than that of $M(n,\alpha)$. It is easy to see that a graph $G$
is in ${\cal B}(n,\alpha,\alpha)$ if and only if $G \in {\cal
B}(n,\alpha)$ and every non-pendant vertex of $G$ has at least one
pendant neighbour. For $i=1,2$ set
$${\cal B}_{i}(n,\alpha,\alpha)=\{G \,|\, G \in {\cal B}(n,\alpha,\alpha)
\,\mbox{and} \, G \in{\cal B}_{i}(n)\}.$$

\vspace{1cm}
\begin{center}
\unitlength 1mm \linethickness{0.4pt}
\begin{picture}(60,-35)

\put(-25,5){\circle*{1}} \put(-25,7){\makebox(0,0)[cc]{$w$}}

\put(-33,5){\circle*{1}}
 \put(-25,-5){\circle*{1}}
\put(-25,-7){\makebox(0,0)[cc]{$w^{\prime}$}}

 \put(-33,-5){\circle*{1}}
\put(-15,0){\circle*{1}} \put(-20,8){\circle*{1}}
\put(-10,8){\circle*{1}}

\put(-17,4){$\cdots$} \put(-20,8.5){$\overbrace{\hspace{1cm}}$}
\put(-15,12){\makebox(0,0)[cc]{$_{2\alpha-n+1}$}}

 \put(-20,-8){\circle*{1}}
\put(-10,-8){\circle*{1}} \put(-20,-16){\circle*{1}}
\put(-10,-16){\circle*{1}} \put(-5,5){\circle*{1}}
\put(-5,7){\makebox(0,0)[cc]{$u$}}

\put(-5,-5){\circle*{1}} \put(3,5){\circle*{1}}
\put(3,-5){\circle*{1}} \qbezier(-15,0)(-15,0)(-25,5)
 \qbezier(-15,0)(-15,0)(-25,-5) \qbezier(-15,0)(-15,0)(-20,8)
  \qbezier(-15,0)(-15,0)(-10,8)\qbezier(-15,0)(-15,0)(-5,5)
 \qbezier(-15,0)(-15,0)(-5,-5) \qbezier(-15,0)(-15,0)(-20,-8)
  \qbezier(-15,0)(-15,0)(-10,-8)\qbezier(-20,-16)(-20,-16)(-20,-8)
  \qbezier(-10,-16)(-10,-16)(-10,-8)\qbezier(-5,-5)(-5,-5)(-5,5) \qbezier(-25,-5)(-25,-5)(-25,5)
\qbezier(-33,5)(-33,5)(-25,5)\qbezier(-33,-5)(-33,-5)(-25,-5)
\qbezier(3,5)(3,5)(-5,5)\qbezier(3,-5)(3,-5)(-5,-5)
\put(-17,-12){$\cdots$} \put(-20,-16.5){$\underbrace{\hspace{1cm}}$}
\put(-13,-19.5){\makebox(0,0)[cc]{$_{n-\alpha-5}$}}
\put(-13,-25){\makebox(0,0)[cc]{ $M_{1}^{\prime}(n,\alpha)$}}

\put(27,0){\circle*{1}}\put(32,8){\circle*{1}}
\put(32,-8){\circle*{1}} \put(35,0){\circle*{1}}
\put(40,8){\circle*{1}} \put(40,-8){\circle*{1}}
\put(45,0){\circle*{1}}
\put(43,10){\circle*{1}}\put(53,10){\circle*{1}}
\put(43,-9){\circle*{1}}\put(53,-9){\circle*{1}}
\put(43,-18){\circle*{1}}\put(53,-18){\circle*{1}}

\qbezier(45,0)(45,0)(35,0)
\qbezier(45,0)(45,0)(40,8)\qbezier(45,0)(45,0)(40,-8)
\qbezier(45,0)(45,0)(43,10) \qbezier(45,0)(45,0)(53,10)
\qbezier(45,0)(45,0)(43,-9)\qbezier(45,0)(45,0)(53,-9)

\qbezier(27,0)(27,0)(35,0)\qbezier(40,-8)(40,-8)(35,0)\qbezier(40,8)(40,8)(35,0)
\qbezier(43,-18)(43,-18)(43,-9)\qbezier(53,-18)(53,-18)(53,-9)
\qbezier(32,8)(32,8)(40,8)\qbezier(32,-8)(32,-8)(40,-8)
\put(45,5){$\cdots$} \put(43,10.5){$\overbrace{\hspace{1cm}}$}
\put(49,13.5){\makebox(0,0)[cc]{$_{2\alpha-n+1}$}}

\put(46,-14){$\cdots$} \put(43,-18.5){$\underbrace{\hspace{1cm}}$}
\put(49,-21.5){\makebox(0,0)[cc]{$_{n-\alpha-4}$}}
\put(43,-26){\makebox(0,0)[cc]{ $M_{1}(n,\alpha)$}}
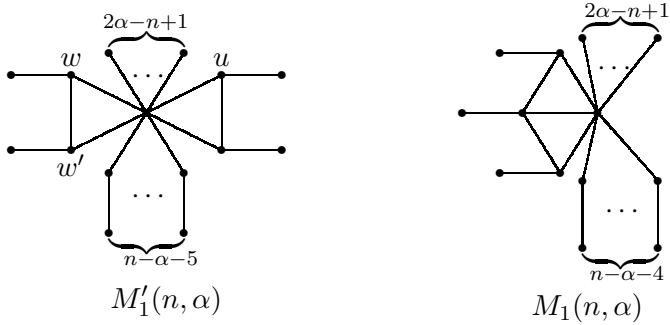
\put(20,-34){\makebox(0,0)[cc]{Fig. 4 \quad the graphs
$M_{1}^{\prime}(n,\alpha)$ and $M_{1}(n,\alpha)$}}

\end{picture}
\end{center}
\vspace{3cm}

\begin{lemma}
\label{juti2} Let $M_{1}(n,\alpha)$ and $M(n,\alpha)$ be the graphs
as shown in Fig. 3 and Fig.4. Then
$\rho(M_{1}(n,\alpha))<\rho(M(n,\alpha))$.
\end{lemma}

\begin{proof}
 Let
\begin{align}
\label{fuction for A} f(x)=x^4-(\alpha+3)x^2-4x+(2\alpha-n+1).
\end{align}By using Lemma \ref{formular1} and tedious calculations we have
\begin{align}
\label{poly for A}
\Phi(M(n,\alpha);x)=x^{2\alpha-n}(x^2-1)^{n-\alpha-2}f(x),
\end{align}
and $\rho(M(n,\alpha))$ is the largest root of the equation
 $f(x)=0$.
 Let
$$ f_{1}(x)=x^8-(\alpha+5)x^6-4x^5-
(n-6\alpha)x^4+4x^3+(4n-9\alpha-5)x^2-2x-(n-2\alpha-1).
$$
We have
$$
\Phi(M_{1}(n,\alpha);x)=x^{2\alpha-n}(x^2-1)^{n-\alpha-4}f_{1}(x),
$$
 and $\rho(M_{1}(n,\alpha))$ is the largest root of the equation
 $f_{1}(x)=0$. It may be verified that
\begin{align}
\label{relation for 1}
f_{1}(x)-(x^2-1)^{2}f(x)=2x[(\alpha-4)x^{3}-2x^{2}+(n-2\alpha)x+2].
\end{align}

For $M_{1}(n,\alpha)$ when $n\geq 10$ we have $\alpha\geq 5$, and
write $\rho(M_{1}(n,\alpha))=\rho$, then $\rho>2$. Then from
(\ref{relation for 1}) we have
\begin{eqnarray*}
-\frac{(\rho^2-1)^{2}}{2\rho}f(\rho)
 &=&(\alpha-4)\rho^{3}-2\rho^{2}+(n-2\alpha)\rho+2\\
&>&(2\alpha-10)\rho^2+(n-2\alpha)\rho+2\\
&>&(n+2\alpha-20)\rho+2>0.
  \end{eqnarray*}
Thus  $f(\rho)<0$, then the largest root of equation $f(x)=0$ is
larger than $\rho$, i.e., $\rho(M(n,\alpha))>\rho(M_{1}(n,\alpha))$.
\end{proof}

\begin{lemma}
\label{largest in 1}
 Graph $M_{1}^{\prime}(n,\alpha)$ alone maximizes the spectral
radius among all the graphs in ${\cal B}_{1}(n,\alpha,\alpha)$.
\end{lemma}

\begin{proof}
Suppose $G^{*}$ is a graph with maximal  spectral radius among the
graphs in ${\cal B}_{1}(n,\alpha,\alpha)$. Write
$\widehat{G^{*}}=B(p,\ell,q)$. Let $x$ be the Perron vector of
$G^{*}$. Now we will prove some properties for $G^{*}$.

\medskip
 {\bf Claim 1.} $\ell=1$.
\medskip

{\bf Proof of Claim 1.} Suppose to the contrary that $\ell\geq2$,
and $v_{1}$, $v_{\ell}$ are the vertices of $G^{*}$ with
$d_{\widehat{G^{*}}}(v_{1})=d_{\widehat{G^{*}}}(v_{\ell})=3$.
Without loss of generality assume that $x_{v_{1}}\geq x_{v_{\ell}}$.
Set
$$N(v_{\ell})=\{v_{\ell}^{\prime},v_{\ell}^{\prime\prime},v_{\ell1}, \cdots, v_{\ell s}\},$$
where $d(v_{\ell}^{\prime})=1$, and $v_{\ell}^{\prime\prime}$ is the
neighbour of $v_{\ell}$ lying on the path between  $v_{1}$ and
$v_{\ell}$. Then $s\geq 2$ follows from the fact that
$d_{G^{*}}(v_{\ell})\geq 4$. Set
$$G^{\prime}=G^{*}-v_{\ell}v_{\ell1}- \cdots- v_{\ell}v_{\ell s}+v_{1}v_{\ell1}+ \cdots+v_{1}v_{\ell s}.$$
Then  $G^{\prime}$ is in ${\cal B}_{1}(n)$ with  $\alpha$ pendant
vertices, and every non-pendant vertex of $G^{\prime}$ has  at least
one pendant neighbour. Thus $G^{\prime}$ is also in ${\cal
B}_{1}(n,\alpha,\alpha)$. While we have
$\rho(G^{\prime})>\rho(G^{*})$. This contradicts the definition of
$G^{*}$.

\medskip
{\bf  Claim 2.} $p=q=3$.
\medskip

{\bf Proof of Claim 2.} Suppose to the contrary that $p\geq4$, and
$uv$ is an edge of the cycle $C_{p}$. Without loss of generality
assume  that $x_{u}\geq x_{v}$. Let $w\,(w\not =u)$ be the neighbour
of $v$ on the cycle $C_{p}$, then $wu\not \in E(G^{*})$. Set
$G^{\prime}=G^{*}-vw+uw.$ Then $G^{\prime}$ is also in ${\cal
B}_{1}(n,\alpha,\alpha)$, while $\rho(G^{\prime})>\rho(G^{*})$.

 Thus from Claim 1 and Claim 2 we have
$\widehat{G^{*}}=B(3,1,3)$. Denote by $v$ the vertex of $G^{*}$ with
$d_{\widehat{G^{*}}}(v)=4$.

\medskip
{\bf  Claim 3.} Every  vertex outside of  $V_{c}(G^{*})$ has degree
at most 2.
\medskip

{\bf Proof of Claim 3.} Suppose to the contrary that there exists a
vertex, say $w$, such that $w\not \in V_{c}(G^{*})$ with
$d_{G^{*}}(w)\geq3$. Let  $w^{\prime}$ be a non-pendant neighbour of
$w$, which does not lie on any path between $v$ and $w$. Let
$v^{\prime},v^{\prime\prime}$ be two neighbours of $v$ on some cycle
of $G^{*}$, and $v^{\prime},v^{\prime\prime}$ do not lie on any path
between $v$ and $w$. Set
$$G^{\prime}= \left\{\begin{array}{ll}
G^{*}-vv^{\prime}-vv^{\prime\prime}+wv^{\prime}+wv^{\prime\prime}, &\, \mbox{if}\, \, x_{w}\geq x_{v}; \\
G^{*}-ww^{\prime}+vw^{\prime}, &\, \mbox{if}\, \, x_{v}> x_{w}.
\end{array}
\right.$$ Then we obtain a graph also in ${\cal
B}_{1}(n,\alpha,\alpha)$ with larger spectral radius than that of
$G^{*}$.

By using the similar arguments as the proof of Claim 3 we may deduce
that every vertex in $V_{c}(G^{*})\backslash\{v\}$ has degree 3.
Thus combining the above results we have
$G^{*}=M_{1}^{\prime}(n,\alpha)$.
\end{proof}

By using the similar proof as that of Lemma \ref{largest in 1}, we
may obtain the following result.

\begin{lemma}
\label{largest in 2}
 Graph $M_{1}(n,\alpha)$ alone maximizes the spectral
radius among all the graphs in ${\cal B}_{2}(n,\alpha,\alpha)$.
\end{lemma}

\begin{theorem}
\label{Bnaa} Let $G$ be any  graph in
 ${\cal B}(n,\alpha,\alpha)$. Then
 $\rho(G)<\rho(M(n,\alpha))$.
\end{theorem}
\begin{proof}
Let $x$ be the Perron vector of $M_{1}^{\prime}(n,\alpha)$, by
symmetry we have $x_{w}=x_{u}$, where $u,v$ are shown in Fig.4. It
is easy to see that
$M_{1}(n,\alpha)=M_{1}^{\prime}(n,\alpha)-ww^{\prime}+uw^{\prime}$.
Then $\rho(M_{1}(n,\alpha))>\rho(M_{1}^{\prime}(n,\alpha))$ follows
from Lemma \ref{deprive l}. Let $G$ be any  graph in
 ${\cal B}(n,\alpha,\alpha)$. Then by using Lemmas \ref{largest in 1}, \ref{largest in 2}
and  \ref{juti2} we have
 $$\rho(G)\leq
 \mbox{max}\{\rho(M_{1}^{\prime}(n,\alpha),\rho(M_{1}(n,\alpha))\}=\rho(M_{1}(n,\alpha))
 <\rho(M(n,\alpha)).$$
 Thus we have  $\rho(G)<\rho(M(n,\alpha))$ for any graph $G$  in
 ${\cal B}(n,\alpha,\alpha)$.
\end{proof}

\subsection{The  graphs  in  ${\cal B}(n,\alpha)$ with $\alpha-1$ pendant vertices}

Set
 $${\cal B}(n,\alpha,\alpha-1)=\{G \,|\, G \in {\cal B}(n,\alpha)
\,\, \mbox{and}\,\, G \,\, \mbox{contains}\,\, \alpha-1
\,\,\mbox{pendant vertices}\}.$$ In this section we will prove that
the spectral radii of the graphs in  ${\cal B}(n,\alpha,\alpha-1)$
are also less than that of $M(n,\alpha)$.  For a graph $G$ in ${\cal
B}(n,\alpha,\alpha-1)$ set $$V^{\prime}(G)=\{v\in V(G) \,|\,
d(v)\geq2 \,\,\mbox{and} \, v \,\mbox{has no pendant neighbour}\},$$
then $|V^{\prime}(G)|\geq 1$. Furthermore $|V^{\prime}(G)|\leq 3$,
for otherwise $\alpha-1$ pendant vertices along with two proper
vertices in $V^{\prime}(G)$ may form an independent set of $G$ with
cardinality $\alpha+1$. Similarly if $|V^{\prime}(G)|=2$, then the
two  vertices in $V^{\prime}(G)$ are incident. And if
$|V^{\prime}(G)|=3$, then the vertices in $V^{\prime}(G)$ lie on a
triangle.

\begin{lemma}
\label{description1}
 Let $G^{*}$ be a graph in ${\cal
B}(n,\alpha,\alpha-1)$ with maximal spectral radius. Then
$|V^{\prime}(G^{*})|\geq 2$, or $\rho(G^{*})<\rho(M(n,\alpha))$.
\end{lemma}

\begin{proof}
If $|V^{\prime}(G^{*})|\geq 2$, the proof is completed. Now suppose
to the contrary that $|V^{\prime}(G^{*})|=1$. Let $u$ be the vertex
in $V^{\prime}(G^{*})$, and $v$, $w$ are two neighbour of $u$. Let
$x$ be the Perron vector of $G^{*}$. Without loss of generality
assume that $x_{v}\geq x_{w}$. Let $w_{1}, \cdots, w_{s}\,(s\geq1)$
be all the pendant neighours of $w$. Set
$$G^{\prime}=G^{*}-ww_{1}- \cdots -ww_{s}+uw_{1}+ \cdots + uw_{s}.$$
If $d_{G^{\prime}}(w)=1$, then every non-pendant vertex of
$G^{\prime}$ has at least one pendant neighour, thus $G^{\prime} \in
{\cal B}(n,\alpha,\alpha)$, and
$$\rho(G^{*})<\rho(G^{\prime})\leq\rho(M(n,\alpha)).$$  If $d_{G^{\prime}}(w)\geq 2$, then
$G^{\prime}$ is also in ${\cal B}(n,\alpha,\alpha-1)$. While
$\rho(G^{\prime})>\rho(G^{*})$. This contradicts the definition of
$G^{*}$.
\end{proof}

\begin{lemma}
\label{D1}
 Let $G$ be a graph in ${\cal B}(n,\alpha,\alpha-1)$. If
$|V^{\prime}(G)|= 2$ and the vertices in $V^{\prime}(G)$ do not lie
on a triangle, then $\rho(G)<\rho(M(n,\alpha))$.
\end{lemma}

\begin{proof}
 Suppose $u,v$ are the  two vertices in $V^{\prime}(G)$.
Let $x$ be the Perron vector of $G$. Without loss of generality
assume that $x_{u}\geq x_{v}$. Let $v_{1}, \cdots, v_{s}\,(s\geq1)$
be all the  neighours of $v$ different from $u$, then $v_{i} \not
\in N(u)$. Set
$$G^{\prime}=G-vv_{1}- \cdots - vv_{s}+uv_{1}+ \cdots + uv_{s}.$$
Then  every non-pendant vertex of  $G^{\prime}$ has at least one
pendant neighour, thus $G^{\prime} \in {\cal B}(n,\alpha,\alpha)$.
While from Lemma \ref{deprive l} and Theorem \ref{Bnaa} we have
$\rho(G)<\rho(G^{\prime})\leq\rho(M(n,\alpha)).$
\end{proof}

\vspace{1cm}
\begin{center}
\unitlength 1mm \linethickness{0.4pt}
\begin{picture}(80,-35)

\put(-25,5){\circle*{1}}

 \put(-25,-5){\circle*{1}}

\put(-15,0){\circle*{1}} \put(-20,8){\circle*{1}}
\put(-10,8){\circle*{1}}

\put(-17,4){$\cdots$} \put(-20,8.5){$\overbrace{\hspace{1cm}}$}
\put(-15,11.5){\makebox(0,0)[cc]{$_{2\alpha-n+1}$}}

 \put(-20,-8){\circle*{1}}
\put(-10,-8){\circle*{1}} \put(-20,-16){\circle*{1}}
\put(-10,-16){\circle*{1}} \put(-5,5){\circle*{1}}

\put(-5,-5){\circle*{1}} \put(3,5){\circle*{1}}
\put(3,-5){\circle*{1}} \qbezier(-15,0)(-15,0)(-25,5)
 \qbezier(-15,0)(-15,0)(-25,-5) \qbezier(-15,0)(-15,0)(-20,8)
  \qbezier(-15,0)(-15,0)(-10,8)\qbezier(-15,0)(-15,0)(-5,5)
 \qbezier(-15,0)(-15,0)(-5,-5) \qbezier(-15,0)(-15,0)(-20,-8)
  \qbezier(-15,0)(-15,0)(-10,-8)\qbezier(-20,-16)(-20,-16)(-20,-8)
  \qbezier(-10,-16)(-10,-16)(-10,-8)\qbezier(-5,-5)(-5,-5)(-5,5) \qbezier(-25,-5)(-25,-5)(-25,5)

\qbezier(3,5)(3,5)(-5,5)\qbezier(3,-5)(3,-5)(-5,-5)
\put(-17,-12){$\cdots$} \put(-20,-16.5){$\underbrace{\hspace{1cm}}$}
\put(-13,-19.5){\makebox(0,0)[cc]{$_{n-\alpha-4}$}}
\put(-10,-30){\makebox(0,0)[cc]{ $M_{2}(n,\alpha)$}}
 \put(32,-8){\circle*{1}}
\put(35,0){\circle*{1}} \put(40,8){\circle*{1}}
\put(40,-8){\circle*{1}} \put(45,0){\circle*{1}}
\put(43,10){\circle*{1}}\put(53,10){\circle*{1}}
\put(43,-9){\circle*{1}}\put(53,-9){\circle*{1}}
\put(43,-18){\circle*{1}}\put(53,-18){\circle*{1}}

\qbezier(45,0)(45,0)(35,0)
\qbezier(45,0)(45,0)(40,8)\qbezier(45,0)(45,0)(40,-8)
\qbezier(45,0)(45,0)(43,10) \qbezier(45,0)(45,0)(53,10)
\qbezier(45,0)(45,0)(43,-9)\qbezier(45,0)(45,0)(53,-9)

\qbezier(40,-8)(40,-8)(35,0)\qbezier(40,8)(40,8)(35,0)
\qbezier(43,-18)(43,-18)(43,-9)\qbezier(53,-18)(53,-18)(53,-9)
\qbezier(32,-8)(32,-8)(40,-8) \put(45,5){$\cdots$}
\put(43,10.5){$\overbrace{\hspace{1cm}}$}
\put(49,13.5){\makebox(0,0)[cc]{$_{2\alpha-n+1}$}}

\put(46,-14){$\cdots$} \put(43,-18.5){$\underbrace{\hspace{1cm}}$}
\put(49,-21.5){\makebox(0,0)[cc]{$_{n-\alpha-3}$}}
\put(40,-30){\makebox(0,0)[cc]{ $M_{3}(n,\alpha)$}}

\put(80,0){\circle*{1}} \put(90,0){\circle*{1}}
\put(85,8){\circle*{1}}\put(85,-8){\circle*{1}}
\put(80,-14){\circle*{1}} \put(90,-14){\circle*{1}}
 \put(80,-20){\circle*{1}}\put(90,-20){\circle*{1}}
 \put(92,-2){\circle*{1}}  \put(92,-12){\circle*{1}}
\qbezier(85,-8)(85,-8)(92,-2) \qbezier(85,-8)(85,-8)(92,-12)
\qbezier(85,-8)(85,-8)(80,-14) \qbezier(85,-8)(85,-8)(90,-14)
\qbezier(85,-8)(85,-8)(80,0)\qbezier(85,-8)(85,-8)(90,0)
\qbezier(80,-20)(80,-20)(80,-14) \qbezier(90,-20)(90,-20)(90,-14)
\qbezier(90,0)(90,0)(80,0)

 \put(90,-9){$\vdots$}

\put(101,-8){\makebox(0,0)[cc]{$_{2\alpha-n+1}$}}

\put(85,-8){$\left.
\begin{array}{ll}
     &  \\
     &  \\
             \end{array}
\right\}    $    }

\qbezier(85,8)(85,8)(80,0) \qbezier(85,8)(85,8)(90,0)
\put(72,0){\circle*{1}}\qbezier(72,0)(72,0)(80,0)
\put(79,2){\makebox(0,0)[cc]{$u$}}
\put(83,-8){\makebox(0,0)[cc]{$v$}}

\put(83,-16){$\cdots$} \put(80,-20.5){$\underbrace{\hspace{1cm}}$}
\put(85,-23.5){\makebox(0,0)[cc]{$_{n-\alpha-3}$}}
\put(85,-30){\makebox(0,0)[cc]{ $M_{3}^{\prime}(n,\alpha)$}}

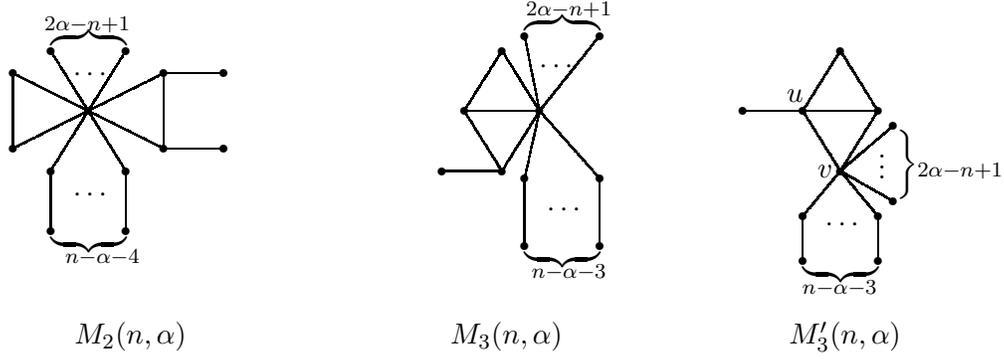
\put(35,-40){\makebox(0,0)[cc]{Fig.5 the graphs $M_{2}(n,\alpha)$,
$M_{3}(n,\alpha)$ and $M_{3}^{\prime}(n,\alpha)$}}
\end{picture}
\end{center}

\vspace{4cm}
\begin{lemma}
\label{juti3}
 Let $M_{2}(n,\alpha)$ and $M_{3}(n,\alpha)$ be the graphs as shown in Fig. 5. Then

 (1). $\rho(M_{2}(n,\alpha))<\rho(M(n,\alpha))$;

(2). $\rho(M_{3}(n,\alpha))<\rho(M(n,\alpha))$.
\end{lemma}

\begin{proof}
Recall that $\rho(M(n,\alpha))$ is the largest root of the equation
$f(x)=0$, where
$$ f(x)=x^4-(\alpha+3)x^2-4x+(2\alpha-n+1).$$
 (1). Let
$$f_{2}(x)=x^6-x^5-
(\alpha+3)x^4+(\alpha-2)x^3-(n-3\alpha-5)x^2+(n-2\alpha+1)x-(n-2\alpha-1).$$
Then we have
$$\Phi(M_{2}(n,\alpha);x)=x^{2\alpha-n}(x^2-1)^{n-\alpha-4}(x^{2}+x-1)f_{2}(x).$$
Thus $\rho(M_{2}(n,\alpha))$ is the largest root of the equation
$f_{2}(x)=0$, and it can be verified that
\begin{align}
\label{relation for 2}
(x^{2}+x-1)f_{2}(x)-(x^2-1)^{2}f(x)=x[(\alpha-2) x^3+(n
-2\alpha)x+2].
\end{align}
For $M_{2}(n,\alpha)$ when $n\geq 10$ we have $\alpha\geq 5$, and
write $\rho(M_{2}(n,\alpha))=\rho$, then
$\rho^{2}>\Delta(M_{2}(n,\alpha))=\alpha+1$, and $\rho>2$. Then from
(\ref{relation for 2}) we have

\begin{eqnarray*}
-\frac{(\rho^2-1)^{2}}{\rho}f(\rho)
 &=&(\alpha-2)\rho^{3}+(n-2\alpha)\rho+2\\
&>&(\alpha-2)(\alpha+1)\rho+(n-2\alpha)\rho+2\\
&=&[\alpha(\alpha-3)+(n -2)]\rho+2>0.
  \end{eqnarray*}
Thus  $f(\rho)<0$, then the largest root of equation $f(x)=0$ is
larger than $\rho$, i.e.,
$\rho(M(n,\alpha))>\rho(M_{2}(n,\alpha))$.

\medskip

(2). Let
 $$ f_{3}(x)=x^8-(\alpha+5)x^6-4x^5-
(n-5\alpha-4)x^4+6x^3+(3n-7\alpha-4)x^2-2x-(n-2\alpha-1),$$ then
$$
\Phi(M_{3}(n,\alpha);x)=x^{2\alpha-n}(x^2-1)^{n-\alpha-4}f_{3}(x).
$$
So $\rho(M_{3}(n,\alpha))$ is the largest root of the equation
$f_{3}(x)=0$, and it may be verified that
\begin{align}
\label{relation for 3}
f_{3}(x)-(x^2-1)^{2}f(x)=x[(\alpha-4)x^3-2x^2+(n-2\alpha +1)x+2].
\end{align}
For $M_{3}(n,\alpha)$ when $n\geq 10$ we have $\alpha\geq 5$, and
write $\rho(M_{3}(n,\alpha))=\rho$, then
$\rho^{2}>\Delta(M_{3}(n,\alpha))=\alpha+1$. Then from
(\ref{relation for 3}) we have

\begin{eqnarray*}
-\frac{(\rho^2-1)^{2}}{\rho}f(\rho)
&=&(\alpha-4)\rho^3-2\rho^2+(n-2\alpha+1)\rho+2\\
&=&(\rho^{3}-2\rho^{2}+(\alpha-5)\rho^3+(n-2\alpha+1)\rho+2\\
&>&(\rho-2)(\alpha+1)+(\alpha-5)(\alpha+1)\rho+(n-2\alpha+1)\rho+2\\
&=&(\alpha^{2}-5\alpha+n-3)\rho-2\alpha\\
 &>&2\alpha(\alpha-6)+2n-6>0.
\end{eqnarray*}
Thus  $f(\rho)<0$, then the largest root of equation $f(x)=0$ is
larger than $\rho$, i.e., $\rho(M(n,\alpha))>\rho(M_{3}(n,\alpha))$.
\end{proof}

Denote by ${\cal B}(n,\alpha,\alpha-1,2)$ the set of the graphs $G$
in ${\cal B}(n,\alpha,\alpha-1)$ with $|V^{\prime}(G)|= 2$ and the
vertices in $V^{\prime}(G)$ lie on a triangle.
\medskip

\begin{lemma}
\label{D2}
 Let $G$ be a graph in ${\cal B}(n,\alpha,\alpha-1,2)$. Then $\rho(G)<\rho(M(n,\alpha))$.
\end{lemma}

\begin{proof}
Let $G^{*}$  be a graph with maximal spectral radius in ${\cal
B}(n,\alpha,\alpha-1,2)$.  First suppose that $G^{*}$ is in ${ \cal
B}_{1}(n)$. Then  $\widehat{G^{*}}=B(3,\ell,q)$. Let $u,w$ be the
two vertices in $V^{\prime}(G^{*})$. By considering some (proper)
coordinates
 of the Perron vector of $G^{*}$ and using Lemma
\ref{deprive l} we may deduce that
$d_{\widehat{G^{*}}}(u)=d_{\widehat{G^{*}}}(w)=2$. And by using the
similar arguments as the proof of Lemma \ref{largest in 1} we have
$\ell=1$ and $q=3$. Let $v$ be the vertex of  $G^{*}$ with
$d_{\widehat{G^{*}}}(v)=4$. Furthermore we have  every  vertex
outside of $V_{c}(G^{*})$ has degree at most 2, and the vertex in
$V_{c}(G^{*})\setminus\{u,v,w\}$ has degree 3, and
$d_{G^{*}}(u)=d_{G^{*}}(v)=2$. Thus we have $G^{*}=M_{2}(n,\alpha)$.
 From Lemma \ref{juti3} we know that $\rho(G^{*})<\rho(M(n,\alpha))$.

Now suppose that $G^{*}$ is in ${ \cal B}_{2}(n)$. Then
$\widehat{G^{*}}=P(0,1,q)$. Similarly as above  we may deduce that
one  vertex   in $V^{\prime}(G^{*})$ has degree 2 in
$\widehat{G^{*}}$. Furthermore  we have  $G^{*}\in
\{M_{3}^{\prime}(n,\alpha), M_{3}(n,\alpha)\}$. Considering the
coordinates $x_{u}$ and $x_{v}$ of the Perron vector $x$ of
$M_{3}^{\prime}(n,\alpha)$ and using Lemma \ref{deprive l} we may
deduce that
$\rho(M_{3}^{\prime}(n,\alpha))<\mbox{max}\{\rho(M_{3}(n,\alpha))\}$.
Combining Lemma \ref{juti3}
 we have
 $$\rho(G)\leq\rho(G^{*})=\mbox{max}\{\rho(M_{3}^{\prime}(n,\alpha)),\rho(M_{3}(n,\alpha))\}=\rho(M_{3}(n,\alpha))<\rho(M(n,\alpha)).$$
Thus we have $\rho(G)<\rho(M(n,\alpha))$ for any graph  $G$ in
${\cal B}(n,\alpha,\alpha-1,2)$.
\end{proof}

\vspace{1cm}
\begin{center}
\unitlength 1mm \linethickness{0.4pt}
\begin{picture}(80,-35)

\put(-25,-5){\circle*{1}}
 \put(-25,-15){\circle*{1}}
\put(-15,-10){\circle*{1}}
 \put(-5,-5){\circle*{1}}
  \put(-10,3){\circle*{1}}
   \put(-10,11){\circle*{1}}
     \put(0,3){\circle*{1}}
     \put(-7,5){$\cdots$} \put(-10,11.5){$\overbrace{\hspace{1cm}}$}
\put(-5,15){\makebox(0,0)[cc]{$_{n-\alpha-4}$}}
   \put(0,11){\circle*{1}}
\qbezier(-5,-5)(-5,-5)(-15,-10)\qbezier(-5,-5)(-5,-5)(-5,-15)
\qbezier(-5,-5)(-5,-5)(-10,3)\qbezier(-5,-5)(-5,-5)(0,3)
\qbezier(-5,-5)(-5,-5)(3,-10)\qbezier(-5,-5)(-5,-5)(3,0)
\qbezier(-25,-5)(-25,-5)(-15,-10)\qbezier(-25,-15)(-25,-15)(-15,-10)\qbezier(-5,-15)(-5,-15)(-15,-10)
\qbezier(-25,-5)(-25,-5)(-25,-15)\qbezier(-5,-15)(-5,-15)(3,-15)
\qbezier(-10,11)(-10,11)(-10,3)\qbezier(0,11)(0,11)(0,3)
\put(-5,-15){\circle*{1}}
 \put(3,-15){\circle*{1}}
 \put(2,-6){$\vdots$}
\put(12,-5){\makebox(0,0)[cc]{$_{2\alpha-n+2}$}} \put(-4,-6){$\left.
\begin{array}{lll}
     &  \\
     &  \\
          \end{array}
\right\}    $    }

 \put(3,0){\circle*{1}}
  \put(3,-10){\circle*{1}}
\put(-10,-30){\makebox(0,0)[cc]{$M_{4}(n,\alpha)$}}

\put(25,5){\circle*{1}} \put(25,-5){\circle*{1}}
\put(35,0){\circle*{1}} \put(45,0){\circle*{1}}
\put(40,8){\circle*{1}}\put(50,8){\circle*{1}}
\put(43.5,4.5){$\cdots$} \put(40,8.5){$\overbrace{\hspace{1cm}}$}
\put(45,12){\makebox(0,0)[cc]{$_{2\alpha-n+2}$}}

\put(40,-8){\circle*{1}}\put(50,-8){\circle*{1}}
\put(40,-16){\circle*{1}}\put(50,-16){\circle*{1}}
\put(55,5){\circle*{1}} \put(55,-5){\circle*{1}}
\put(61,5){\circle*{1}} \put(61,-5){\circle*{1}}
\qbezier(45,0)(45,0)(35,0)\qbezier(45,0)(45,0)(40,8)

\qbezier(45,0)(45,0)(50,8) \qbezier(45,0)(45,0)(40,-8)
\qbezier(45,0)(45,0)(50,-8) \qbezier(45,0)(45,0)(55,5)
\qbezier(45,0)(45,0)(55,-5) \qbezier(40,-16)(40,-16)(40,-8)
\qbezier(50,-16)(50,-16)(50,-8)

\qbezier(25,5)(25,5)(35,0)\qbezier(25,-5)(25,-5)(35,0)

\qbezier(25,5)(25,5)(25,-5)\qbezier(55,5)(55,5)(55,-5)

\put(40,-30){\makebox(0,0)[cc]{ $M_{5}(n,\alpha)$}}
\qbezier(61,-5)(61,-5)(55,-5)\qbezier(61,5)(61,5)(55,5)
\put(43.5,-14){$\cdots$} \put(40,-16.5){$\underbrace{\hspace{1cm}}$}
\put(45,-20){\makebox(0,0)[cc]{$_{n-\alpha-5}$}}
\put(80,0){\circle*{1}} \put(90,0){\circle*{1}}
\put(85,8){\circle*{1}}\put(85,-8){\circle*{1}}
\put(80,-14){\circle*{1}} \put(90,-14){\circle*{1}}
 \put(80,-20){\circle*{1}}\put(90,-20){\circle*{1}}
 \put(92,-2){\circle*{1}}  \put(92,-12){\circle*{1}}
\qbezier(85,-8)(85,-8)(92,-2) \qbezier(85,-8)(85,-8)(92,-12)
\qbezier(85,-8)(85,-8)(80,-14) \qbezier(85,-8)(85,-8)(90,-14)
\qbezier(85,-8)(85,-8)(80,0)\qbezier(85,-8)(85,-8)(90,0)
\qbezier(80,-20)(80,-20)(80,-14) \qbezier(90,-20)(90,-20)(90,-14)
\qbezier(90,0)(90,0)(80,0)
 \put(90,-9){$\vdots$}

\put(101,-8){\makebox(0,0)[cc]{$_{2\alpha-n+2}$}}
\put(85,-8){$\left.
\begin{array}{ll}
     &  \\
     &  \\
             \end{array}
\right\}    $    }

\qbezier(85,8)(85,8)(80,0) \qbezier(85,8)(85,8)(90,0)
\put(79,2){\makebox(0,0)[cc]{$u$}}
\put(83,-8){\makebox(0,0)[cc]{$v$}}

\put(83,-16){$\cdots$} \put(80,-20.5){$\underbrace{\hspace{1cm}}$}
\put(85,-23.5){\makebox(0,0)[cc]{$_{n-\alpha-3}$}}
\put(90,-30){\makebox(0,0)[cc]{ $M_{6}(n,\alpha)$}}

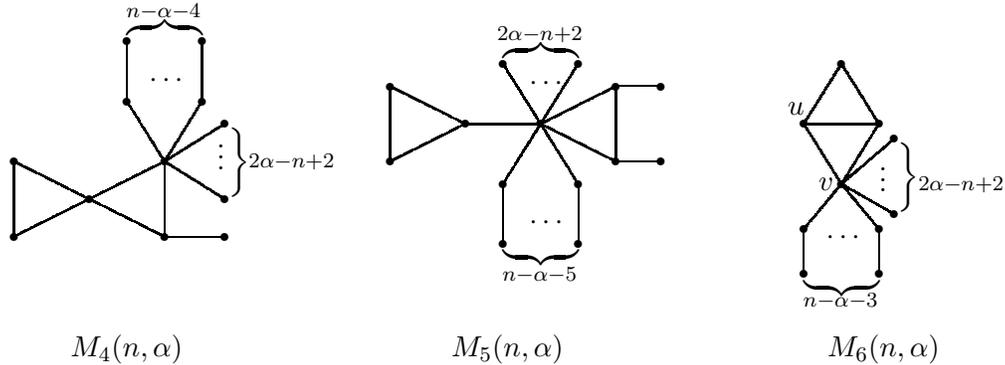
\put(35,-40){\makebox(0,0)[cc]{Fig.6 \quad the graphs
$M_{4}(n,\alpha)$, $M_{5}(n,\alpha)$ and $M_{6}(n,\alpha)$}}
\end{picture}
\end{center}
\vspace{4cm}

\begin{lemma}
\label{juti} Let $M_{i}(n,\alpha)$ be the graph as shown in Fig.6.
Then we have  $\rho(M_{i}(n,\alpha))<\rho(M(n,\alpha))$ for each
$i=4,5,6$.
\end{lemma}

\begin{proof}
 Recall that $\rho(M(n,\alpha))$ is the largest root of the equation $f(x)=0$, where
$$ f(x)=x^4-(\alpha+3)x^2-4x+(2\alpha-n+1).$$

(1). Let
$$ f_{4}(x)=x^7-(\alpha+5)x^5-4x^4-
(n-6\alpha-3)x^3+2(\alpha+1)x^2+(4n-8\alpha-9)x+2(n-2\alpha-2),
$$
then
$$\Phi(M_{4}(n,\alpha);x)=x^{2\alpha-n+1}(x^2-1)^{n-\alpha-4}f_{4}(x),$$
and $\rho(M_{4}(n,\alpha))$ is the largest root of the equation
$f_{4}(x)=0$. And it may be verified that
 \begin{align}
 \label{relation for 4}
xf_{4}(x)-(x^2-1)^{2}f(x)=(2\alpha-5)x^{4}+2(\alpha-3)
x^3+(2n-3\alpha-4)x^{2}+2(n-2\alpha)x+(n-2\alpha-1).
\end{align}
 For
$M_{4}(n,\alpha)$ when $n\geq10$ we have $\alpha \geq5$, and write
$\rho(M_{4}(n,\alpha))=\rho$, then
$\rho^{2}>\Delta(M_{3}(n,\alpha))=\alpha$. Then from (\ref{relation
for 4}) we have
\begin{eqnarray*}
-(\rho^2-1)^{2}f(\rho)
 &=&(2\alpha-5)\rho^{4}+(2\alpha-6)
\rho^3+(2n-3\alpha-4)\rho^{2}+(2n-4\alpha)\rho+(n-2\alpha-1)\\
&>&[(2\alpha-5)\alpha+(2n-3\alpha-4)]\rho^{2}+[(2\alpha-6)\alpha
+(2n-4\alpha)]\rho+(n-2\alpha-1)\\
&>&(2\alpha-5)\alpha+(2n-3\alpha-4)+(n-2\alpha-1)\\
&=&2\alpha^{2}-10\alpha+3n-5>0.
  \end{eqnarray*}
Thus  $f(\rho)<0$, then the largest root of equation $f(x)=0$ is
larger than $\rho$, i.e., $\rho(M(n,\alpha))>\rho(M_{4}(n,\alpha))$.

\medskip

(2). Let $$f_{5}(x)=x^7-3x^{6}-\alpha x^5+3(\alpha+1)x^4-
(n-\alpha-4)x^3+(3n-8\alpha-7)x^2-(n-2\alpha-3)x-2(n-2\alpha-2),$$
then we have
$$
\Phi(M_{5}(n,\alpha);x)=x^{2\alpha-n+1}(x^2-1)^{n-\alpha-6}(x+1)^{2}(x^{2}+x+1)f_{5}(x).$$
Thus $\rho(M_{3}(n,\alpha))$ is the largest root of the equation
$f_{5}(x)=0$, and it may be verified that
\begin{eqnarray*}
\label{relation for 5}
&&x(x^{2}+x+1)f_{5}(x)-(x+1)^{2}(x^2-1)^{2}f(x)\\
&&=(2\alpha-4)x^{6}-(2\alpha-6)x^{5}+(2n-6\alpha+5)x^4-(2n-4\alpha)x^{3}-(2n-5\alpha+3)x^{2}+2x+(n-2\alpha-1).
  \end{eqnarray*}
 For $M_{5}(n,\alpha)$
when $n\geq10$ we have $\alpha \geq5$, and write
$\rho(M_{3}(n,\alpha))=\rho$, then $\rho>2$. Then from
(\ref{relation for 5}) we have
\begin{eqnarray*}
&&-(\rho+1)^{2}(\rho^2-1)^{2}f(\rho)\\
 &=&(2\alpha-4)\rho^{6}-(2\alpha-6)\rho^{5}+(2n-6\alpha+5)\rho^4-(2n-4\alpha)\rho^{3}-(2n-5\alpha+3)\rho^{2}+2\rho+(n-2\alpha-1)\\
&>&(2\alpha-2)\rho^{5}+(2n-6\alpha+5)\rho^4-(2n-4\alpha)\rho^{3}-(2n-5\alpha+3)\rho^{2}+2\rho+(n-2\alpha-1)\\
&>&(2n-2\alpha+1)\rho^4-(2n-4\alpha)\rho^{3}-(2n-5\alpha+3)\rho^{2}+2\rho+(n-2\alpha-1)\\
&>&(2n+2)\rho^{3}-(2n-5\alpha+3)\rho^{2}+2\rho+(n-2\alpha-1)\\
&>&(2n+5\alpha+1)\rho^{2}+2\rho+(n-2\alpha-1)>0.
  \end{eqnarray*}
Thus  $f(\rho)<0$, then the largest root of equation $f(x)=0$ is
larger than $\rho$, i.e., $\rho(M(n,\alpha))>\rho(M_{5}(n,\alpha))$.

\medskip

(3). Let
$$
f_{6}(x)=x^5-2x^4-(\alpha+2)x^3+2(\alpha+1)x^2-(n-2\alpha-2)x+(2n-4\alpha-4),$$
then  we have
$$
\Phi(M_{6}(n,\alpha);x)=x^{2\alpha-n+1}(x^2-1)^{n-\alpha-4}(x+1)^2f_{6}(x),
$$
and $\rho(M_{6}(n,\alpha))$ is the largest root of the equation
$f_{6}(x)=0$. And it may be verified that
\begin{align}
\label{relation for6}
x(x+1)^2f_{6}(x)-(x^2-1)^{2}f(x)=(\alpha-4)x^2+2
x+(n-2\alpha-1).
\end{align}
Then from (\ref{relation for6}) we have
\begin{eqnarray*}
-(\rho^2-1)^{2}f(\rho) &=&(\alpha-4)\rho^2+2\rho+(n-2\alpha-1)\\
&>& (\alpha-4)(\alpha+1)+(n-2\alpha-1)\\
&=&\alpha^{2}-5\alpha-5>0.
  \end{eqnarray*}
Thus  $f(\rho)<0$, then the largest root of equation $f(x)=0$ is
larger than $\rho$, i.e., $\rho(M(n,\alpha))>\rho(M_{6}(n,\alpha))$.
\end{proof}

Denote by ${\cal B}(n,\alpha,\alpha-1,3)$ the set of the graphs $G$
in ${\cal B}(n,\alpha,\alpha-1)$ with $|V^{\prime}(G)|= 3$.
\medskip

\begin{lemma}
\label{D3}
 Let $G$ be a graph in ${\cal B}(n,\alpha,\alpha-1,3)$. Then
$\rho(G)<\rho(M(n,\alpha))$.
\end{lemma}

\begin{proof}
Let $G^{*}$ be a graph with maximal spectral radius in ${\cal
B}(n,\alpha,\alpha-1,3)$. First suppose that $G^{*}$ is in ${ \cal
B}_{1}(n)$. Then $\widehat{G^{*}}=B(3,\ell,q)$. By using the similar
arguments as the proofs of  Lemma \ref{largest in 1} we may deduce
that $\ell\leq 2$ and $q=3$. Furthermore we have $G^{*}\in
\{M_{4}(n,\alpha),M_{5}(n,\alpha)\}$. Combining Lemma \ref{juti} we
have
 $$\rho(G)\leq\rho(G^{*})=\mbox{max}\{\rho(M_{4}(n,\alpha)),\rho(M_{5}(n,\alpha))\}<\rho(M(n,\alpha)).$$

Now suppose that $G^{*}$ is in ${ \cal B}_{2}(n)$. Then  we have
$G^{*}=M_{6}(n,\alpha)$. By  Lemma \ref{juti}
 we have
 $$\rho(G)\leq\rho(G^{*})=\rho(M_{6}(n,\alpha))<\rho(M(n,\alpha)).$$
 Thus we have $\rho(G)<\rho(M(n,\alpha))$ for any graph  $G$ in
${\cal B}(n,\alpha,\alpha-1,3)$.
\end{proof}

Combining the results of  Lemmas \ref{D1}, \ref{D2} and \ref{D3} we
have the following result.
\begin{theorem}
\label{Bnaa-1} Let $G$ be any  graph in
 ${\cal B}(n,\alpha,\alpha-1)$. Then
 $\rho(G)<\rho(M(n,\alpha))$.
\end{theorem}

\end{document}